\documentclass[bpam,reqno]{ipart}

\usepackage{graphicx,color} 
\usepackage{CJKutf8}

\newtheorem{theorem}{Theorem}[section]
\newtheorem{proposition}{Proposition}[section]
\newtheorem{corollary}{Corollary}[section]

\startlocaldefs
\endlocaldefs

\pubyear{2024}
\volume{1}
\issue{3}
\firstpage{1}
\lastpage{7}
\begin{document}

\begin{frontmatter}

\title[The DG and CG discretizations of the derivative coincide]{The discretizations of the derivative by the continuous Galerkin and the discontinuous Galerkin methods are
exactly the same}

\begin{aug}
    \author{\fnms{Bernardo} \snm{Cockburn}\thanksref{t2}\ead[label=e1]{bcockbur@umn.edu}}
    \address{School of Mathematics, University of Minnesota, MN 55455\\
             USA\\
             \printead{e1}}
             \thankstext{t2}{Bernardo Cockburn's research was supported in part by the Advanced Computational Center for Entry Systems Simulation (ACCESS) through NASA grant 80NSSC21K1117.}
\end{aug}
\received{\sday{28} \smonth{3} \syear{2023}}

\begin{abstract}
In the framework of ODEs, 
we uncover a new  link between the 
continuous Galerkin method (see Math. Comp. (1972), 26 (118 and 120), 415-426 and 881-891) and the discontinuous Galerkin method (see Mathematical Aspects of Finite elements in PDEs, (1974), 89-123), namely, that the discretizations of the derivative by these two methods are the same. A direct consequence of this result is the construction of a new elementwise post-processing of the approximate solution provided by the Discontinuous Galerkin method. 
When the DG method uses polynomials of degree $k\ge0$, the post-processing consists in adding, to the DG approximate solution, the (scaled) left-Radau polynomial of 
degree $k+1$ multiplied by the jump of the approximate solution at the left boundary of the interval. No extra computation is required. The resulting new approximation is continuous and, for  $k>0$,
converges with order $k+2$, that is, with one order more than the original discontinuous Galerkin approximation. For $k=0$, the  order remains the same.
\end{abstract}


%%  Upper case for every keyword

%\tableofcontents
\centerline{\small This paper is dedicated to Chi-Wang Shu in the occasion of his 65th birthday} 

\end{frontmatter}

%%  The body

\section{Prologue}
I met Chi-Wang in August 1986 in Minneapolis.
On a Friday afternoon, we coincided at the 
office of the Institute for Mathematics and 
its Applications of the University of 
Minnesota, where we were to spend a year as 
postdocs. We were surprised that we both had 
been working on numerical methods for 
nonlinear hyperbolic conservation laws but I 
suppose that the organizers of the I.M.A. 
86-87 program were to blame for that. Strangely enough, the 
full year 
passed without us having any academic 
interaction of consequence. However, when the time was almost over, a conversation, totally unexpectedly, led to the uncovering of the (Runge-Kutta)
Discontinuous Galerkin (DG) methods for 
nonlinear hyperbolic conservation laws (for 
the one-space dimension piecewise linear case!). The rest is 
history.

As we all know, the work on this rich  subject is still vigorously going on. And perhaps,
after almost four decades of writing papers, 
reviews, monographs, of organinzing meetings, 
and of training Ph.D. students and postdocs, this is a good 
occasion to pause to celebrate the adventure. 

\section{Introduction}
\label{sec:introduction}
Here we do that by going full circle and revisiting the original DG method for ordinary differential equations (ODEs). Let us begin by recalling the introduction of the first finite element methods for ODEs.
In 1972, Hulme \cite{HulmeI72,HulmeII72} introduced the continuous Galerkin (CG) method for solving ODEs. Two  years later, Lesaint and Raviart \cite{LesaintRaviart74} studied the application of the DG method introduced in 1973 by Reed and Hill \cite{ReedHill73} for the linear hyperbolic equations of neutron transport, to ODEs. The $hp$-version of the resulting DG method was analyzed in \cite{SchoetzauSchwabCALCOLO} and some extensions considered in \cite{DelfourHagerTrochu81}.

Although fifty years have passed, the relation between these two methods seems to have remained unexplored, to the best knowledge of the author. Did this happen, because the CG method has a continuous approximation whereas the DG method does not, or because the CG method does not display an obvious stabilization whereas the DG does, or because the CG method converges with order $2k$ at the nodes whereas the DG method converges with order $2k+1$, when both methods use polynomials of degree $k$? No one can tell with certainty. { And yet,} and remarkably enough, these two methods are closer than the above-mentioned differences seem to suggest. 

Indeed, here we show that the CG method with polynomials of degree $k+1$ and the DG method with polynomials of degree $k$ provide the {\bf very same discretization of the time derivative}. 
This surprising result allows us to introduce a
 simple elementwise post-processing of the DG approximate solution
 which is continuous, displays a significantly smaller error, and has  one additional order of accuracy when polynomials of degree $k>0$ are used. The post-processing is found by applying the technique of  transforming spaces into stabilizations recently proposed in \cite{Cockburn23} to the weak formulation defining the DG method.

The rest of the paper is organized as follows. In Section 3, we state, prove and discuss our main result.  In Section 4, we briefly  describe our ongoing and future work. We end in Section 5, with some words on the good old times and Chinese food.

\section{The main result}

{\bf The CG and DG methods}. The continuous Galerkin \cite{HulmeI72,HulmeII72} method to numerically approximate the solution of the model ODE
\begin{alignat*}{3}
\frac{d}{dt} u(t) =& f(t,u(t)) &&\quad  \forall \; t\in (0,T),
&&\qquad
u(0)=u_0
\end{alignat*}
defines its approximation as follows. { Let $0=:t_0<t_1<\dots<t_N:=T$ be a partition of the interval $[0,T]$. Then,
on} the interval $I_n:=(t_{n-1},t_n)$, the approximation $u^{{\rm CG}}_h$ is the element of the space of polynomials of degree {$k+1$, $\mathcal{P}_{k+1}(I_n)$, and $k\ge 0,$} which solves the weak formulation
\begin{alignat*}{2}
\int_{I_n}  v(t)\frac{d}{dt} u^{{\rm CG}}_h (t) \,dt =&\;
\int_{I_n}  f(t,u^{{\rm CG}}_h(t)) \,v(t)\,dt &&\quad\forall v\in \mathcal{P}_k(I_n),
\\
u^{{\rm CG}}_h(t_{n-1}^+)=&\;u^{{\rm CG}}_h(t_{n-1}^-),
\end{alignat*}
where $u^{{\rm CG}}_h(0^-):=u_0$. In our current vocabulary, this is a Petrov-Galerkin method because the test and trial spaces, $\mathcal{P}_k(I_n)$ and $\mathcal{P}_{k+1}(I_n)$, respectively, are not the same.

The DG method \cite{LesaintRaviart74,ReedHill73} defines its approximation in the following manner.
On the interval $I_n:=(t_{n-1},t_n)$, the approximation $u^{{\rm DG}}_h$ is the element of the space of polynomials of degree $k\ge0$ $\mathcal{P}_k(I_n)$ which solves, in our notation, the weak formulation
\begin{alignat*}{2}
-\int_{I_n} u^{{\rm DG}}_h (t) \frac{d}{dt} v(t)\,dt + \widehat{u}^{\,{\rm DG}}_h v |_{t_{n-1}}^{t_n}=&
\int_{I_n}  f(t,u^{{\rm DG}}_h(t)) \,v(t)\,dt &&\quad\forall v\in \mathcal{P}_k(I_n),
\\
{\widehat{u}^{\,{\rm DG}}_h(t_{n-1})}=&\;u^{{\rm DG}}_h(t_{n-1}^-),
\end{alignat*}
where $u^{{\rm DG}}_h(0^-):=u_0$. The function $\widehat{u}^{\,{\rm DG}}_h$ is nowadays usually called {\bf upwinding numerical trace}. { It is defined only at the nodal points $t_m$ but, to alleviate the notation, we
write that $\widehat{u}^{\,{\rm DG}}_h(t_m^+):=\widehat{u}^{\,{\rm DG}}_h(t_m^-):=\widehat{u}^{\,{\rm DG}}_h(t_m)$,  for $m=0,\dots,N$.}

{This completes the definition of the DG method. The DG method is what we now consider to be a genuine Galerkin method, as the test and trial spaces are the same, $\mathcal{P}_k(I_n)$.}

{\bf Scaled left-Radau polynomials and their properties}. To relate these two methods, we are going to need to introduce the so-called left-Radau polynomials. The left-Radau polynomial of degree $k+1$,
is defined as
\[
R_{k+1}(\tau) =\frac12 P_{k+1}(-1)\,P_{k+1}(\tau)
              +\frac12 P_{k}(-1)  \,P_{k}(\tau)
              \qquad\forall\;\tau\in [-1,1],
\]
where $P_\ell$ is the Legendre polynomial of degree $\ell$.
On the interval $I_n$, we define the corresponding {\em scaled} left-Radau polynomial by
\begin{alignat*}{2}
\mathcal{R}_{n,k+1}(t)&:=R_{k+1}(\tau_n(t))
&&\qquad \forall\;t\in I_n,
\\
\tau_n(t)&=(2t-(t_n+t_{n-1}))/(t_n-t_{n-1})
&&\qquad
\forall\;t\in I_n.
\end{alignat*}
Thanks to the properties of the Legendre polynomials and the scaling function $\tau_n: I_n\mapsto (-1,1)$, it is easy to see that
$\mathcal{R}_{n,k+1} $ is the only polynomial of degree $k+1$ such that
\begin{alignat*}{2}
 {\text{\rm(i)}}& \;\;\;\;\;\;\;\;\; \mathcal{R}_{n,k+1}(t_n)\;\;\;\;\;&&=0,
 \\
 {\text{\rm(ii)}}&  \;\;\;\;\;\;\;\;\;\mathcal{R}_{n,k+1}(t_{n-1})&&=1,
 \\
 {\text{\rm(iii)}}&  \;\;\int_{I_n}\mathcal{R}_{n,k+1}(t)\,w(t)\,dt&&=0 \quad\forall w\in \mathcal{P}_{k-1}(I_n).
 \end{alignat*}

 {\bf The elementwise post-processing}.
 With this left-Radau polynomial, we define the elementwise post-processing of $u_h^{{\rm DG}}$ by
\[
 u^{{\rm DG},*}_h(t):=u^{{\rm DG}}_h(t)+(\widehat{u}^{\,{\rm DG}}_h-u^{{\rm DG}}_h)(t^+_{n-1}) \; \mathcal{R}_{n,k+1}(t) \quad\forall t\in I_n.
 \]
 With this simple post-processing, we are going to be able to establish the stated relation between the CG and DG methods. Before doing that, let us gather some properties of this function in the following result.
\begin{proposition}
\label{proposition3.1}
We have that
 \begin{alignat*}{2}
 {\text{\rm(a)}}& \;\; u^{{\rm DG},*}_h \in \mathcal{P}_{k+1}(I_n),
 \\
 {\text{\rm(b)}}&  \;\; u^{{\rm DG},*}_h \in \mathcal{C}^0[0,T],
 \\
 {\text{\rm(c)}}&  \;\; u^{{\rm DG},*}_h(t_\ell)=\widehat{u}^{\,{\rm DG}}_h(t_\ell).
 \end{alignat*}
 \end{proposition}
 \begin{proof}
Property (a) follows from the fact that 
$u^{{\rm DG}}_h\in\mathcal{P}_k(I_n)$ and that  {$\mathcal{R}_{n,k+1} \in \mathcal{P}_{k+1}(I_n)$.} Property (b) follows from Property (c), and Property(c) can be proven as follows:

\small{
\begin{alignat*}{6}
&u^{{\rm DG},*}_h(t^+_{n-1})&&= u^{{\rm DG}}_h(t^+_{n-1})&&+(\widehat{u}^{\,{\rm DG}}_h-u^{{\rm DG}}_h)(t^+_{n-1}) \; \mathcal{R}_{n,k+1}(t^+_{n-1}) &&={\widehat{u}^{\,{\rm DG}}_h(t_{n-1})},\text{ by (ii),}
\\
&u^{{\rm DG},*}_h(t^-_{n})&&= u^{{\rm DG}}_h(t^-_{n})&&+(\widehat{u}^{\,{\rm DG}}_h-u^{{\rm DG}}_h)(t^+_{n-1}) \; \mathcal{R}_{n,k+1}(t^-_{n}) &&={u}^{{\rm DG}}_h(t^-_{n})= \widehat{u}^{\,{\rm DG}}_h(t_{n}),
\end{alignat*}
}

\noindent by (i). This completes the proof.
\end{proof}

{\bf The rewriting of the formulation of the DG method}.
We are now ready to state our main result.

\begin{theorem}
\label{theorem3.1}
With the notation introduced above, we have that
\begin{alignat*}{2}
\int_{I_n}v(t)\; \frac{d}{dt} u^{{\rm DG},*}_h(t)\;dt =&
\int_{I_n}  f(t,u^{{\rm DG}}_h(t)) \,v(t)\,dt &&\quad\forall v\in \mathcal{P}_k(I_n),
\\
u^{{\rm DG},*}_h(t_{n-1}^+)=&\;u^{{\rm DG},*}_h(t_{n-1}^-),
\end{alignat*}
where $u^{{\rm DG},*}_h(0^-):=u_0$.
\end{theorem}
{ Note that the new formulation involves a {\it post-processing} of the DG solution, $u_h^{\rm DG,*}$, which incorporates  the DG approximation itself $u_h^{\rm DG}$ {\it as well as} its numerical trace $\widehat{u}^{\,{\rm DG}}_h$ into a single function. Taking this into account, we can say that this} result states that the only difference between the DG and the CG methods is how the right-hand side of the ODE is evaluated. {\bf Both discretizations of the derivative are identical.}  In particular, when $f$ depends only on $t$, the approximations provided {by the CG, $u_h^{\rm CG}$, and DG, $u_h^{\rm DG}$, methods} {\em coincide}. A more general result is described in the following result.

\begin{corollary}
\label{corollary3.1}
Assume that the CG method uses polynomials of degree $k+1$ and that the DG method uses polynomials of degree $k$. Assume also that the mapping $t\mapsto f(t,u(t))$ is approximated by a Lagrange interpolation $\mathcal{I}_R f(\cdot,u(\cdot))$ at the zeros of the left-Radau polynomial
$\mathcal{R}_{n,k+1}$. Then,
\begin{itemize}
    \item [{\rm (1)}] $u_h^{{\rm DG},*}=u_h^{{\rm CG}}$ on $I_n$.
   \item [{\rm (2)}] The post-processing $u_h^{{\rm DG},*}$ converges with order  $k+2$ on $I_n$, when $k>0$.
      \item [{\rm (3)}]  The DG approximation $u^{{\rm DG}}_h$ {\em super-converges} with order $k+2$ at the $k+1$ zeros of the left-Radau polynomial $\mathcal{R}_{n,k+1}$. 
\end{itemize}
\end{corollary}
\begin{proof} Property (1) follows from Theorem \ref{theorem3.1} and the fact that
{\small $\mathcal{I}_R f(\cdot,u_h^{{\rm DG}}(\cdot))$ conicides with $
\mathcal{I}_R f(\cdot,u_h^{{\rm DG},*}(\cdot))$}. Property (2) follows from Property (1) and  Hulme's result \cite[Theorem 2]{HulmeII72}. Finally, Property (3) follows from (1), (2) and the fact that, by construction,
$u_h^{{\rm DG}}(t_\ell)=u_h^{{\rm DG},*}(t_\ell)$
at the zeros $t_\ell$ of the left-Radau polynomial $\mathcal{R}_{n,k+1}$. This completes the proof.
\end{proof}

{\bf Proof of the main result}.
 We end this section with a proof of our main result.
To do that, we follow \cite{Cockburn23} and, roughly speaking, transform the stabilization term of the DG method into a space. We proceed as follows. First, we integrate by parts to get, for any $v \in \mathcal{P}_k(I_n)$, that
\begin{alignat*}{2}
\int_{I_n}  v(t)\;\frac{d}{dt} u^{{\rm DG}}_h(t)\;dt + S_h(v)=
\int_{I_n}  f(t, u^{\,{\rm DG}}_h(t)) \,v(t)\,dt,
\end{alignat*}
where,
\begin{alignat*}{2}
 S_h(v)&:=(\widehat{u}^{\,{\rm DG}}_h-u^{{\rm DG}}_h) v |_{t_{n-1}}^{t_n}
=-(\widehat{u}^{\,{\rm DG}}_h-u^{{\rm DG}}_h)(t^+_{n-1})\; v(t_{n-1}^+),
 \end{alignat*}
 is the so-called {\em stabilization} term.
 Now, we transform it as follows:
 \begin{alignat*}{2}
 S_h(v)&=-(\widehat{u}^{\,{\rm DG}}_h-u^{{\rm DG}}_h)(t^+_{n-1})\; v(t_{n-1}^+)
 \\
 &=(\widehat{u}^{\,{\rm DG}}_h-u^{{\rm DG}}_h)(t^+_{n-1})\; \big (0\cdot v(t_n^-)-1\cdot v(t_{n-1}^+)\big )
 \\
 &=(\widehat{u}^{\,{\rm DG}}_h-u^{{\rm DG}}_h)(t^+_{n-1})\; \big (\mathcal{R}_{n,k+1}\,v|_{t_{n-1}}^{t_n}\big )&&\quad\text{ by (i) and (ii)},
 \\
 &=(\widehat{u}^{\,{\rm DG}}_h-u^{{\rm DG}}_h)(t^+_{n-1})\; \int_{I_n} \frac{d}{dt}\big(\mathcal{R}_{n,k+1}\,v\big)(t)\; dt
 \\
 &=(\widehat{u}^{\,{\rm DG}}_h-u^{{\rm DG}}_h)(t^+_{n-1}) \; \int_{I_n}  v(t)\;\frac{d}{dt}\mathcal{R}_{n,k+1}(t) \;dt,
\end{alignat*}
by Property (iii) of the polynomial $\mathcal{R}_{n,k+1}$. This completes the proof.

\section{Ongoing and future work}
\label{sec:ongoingwork}
We have started to extend these results to the linear transport equation,
a necessary stepping stone towards extending them to the compressible 
Navier-Stokes equations in three-space dimensions, see \cite{LalCandlerCockburn23}.

\section{Epilogue}
The times in which I worked with Chi-Wang have been the most exciting academic moments I experienced. When we get 
together, we usually refer to them as the {\it good old times}. 
My only regret, however, is that I never convinced Chi-Wang that the best Chinese food is the Peruvian Chinese food, as it is an absolutely delicious fusion of Cantonese and Peruvian cuisines. I grew up in Per\'u, where Chinese restaurants have a special name, namely, CHIFA, which is the Spanish version of Ch\={\i} M\u{\i} F\`an, \begin{CJK*}{UTF8}{gbsn}
吃米饭
\end{CJK*} ,
%吃  米  饭 , 
as Chi-Wang himself discovered. If we were in Per\'u, I would bring Chi-Wang to a CHIFA to celebrate his birthday: {\it Happy 65-th birthday, Chi-Wang!}

% The next command determines the bibliography style. Please do not
% change this.

\section*{Acknowledements} I would like to thank Sigal Gottlieb for the invitation to contribute to this special issue, to Yanlai Chen for his cool feedback and his sharp proofreading, and to Shukai Du for providing the Chinese characters for {\it eating}. {I  will also like to thank two anonymous referees for their careful reading of the paper.} 

%\bibliographystyle{plain}
%\bibliography{CRAS.bib}

\end{document}